\documentclass[10pt,a4paper]{amsart}%
\usepackage{amsfonts,amsmath,amssymb}
\usepackage{hyperref}
\usepackage[all]{xy}

\newtheorem*{theorem*}{Theorem}
\newtheorem{lemma}{Lemma}[subsection]
\newtheorem{proposition}[lemma]{Proposition}
\newtheorem{remark}[lemma]{Remark}
\newtheorem{example}[lemma]{Example}
\newtheorem{theorem}[lemma]{Theorem}
\newtheorem{definition}[lemma]{Definition}
\newtheorem{notation}[lemma]{Notation}

\newtheorem{corollary}[lemma]{Corollary}

\newtheorem*{conjecture*}{Conjecture}

\baselineskip 24pt \textwidth 15 cm  \theoremstyle{plain}
\newcommand{\tr}{\operatorname{Tr}}

\newcommand{\Hom}{\operatorname{Hom}}

\newcommand{\cc}{\mathbb{C}}

\newcommand{\rr}{\mathbb{R}}

\newcommand{\re}{\operatorname{Re}}

\newcommand{\Z}{{\mathbb Z}}

\newcommand{\R}{{\mathbb R}}
\newcommand{\C}{{\mathbb C}}

\newcommand{\Sc}{{\mathcal S}}
\newcommand{\G}{{\mathcal G}}

\newcommand{\Fre}{{Fr\'{e}chet \,}}

\newcommand{\Fou}{{\mathcal{F}}}

\newcommand{\g}{{\mathfrak{g}}}
\newcommand{\h}{{\mathfrak{h}}}

\newcommand{\Supp}{\mathrm{Supp}}

\newcommand{\gd}{\g^{\sigma}}

\newcommand{\de}{def}
\newcommand{\Lie}{\operatorname{Lie}}
\renewcommand{\sl}{\operatorname{sl}}

\begin{document}

\author{Avraham Aizenbud}
\address{Faculty of Mathematics
and Computer Science, The Weizmann Institute of Science POB 26,
Rehovot 76100, ISRAEL.} \email{aizenr@yahoo.com}

\title[An analog of integrability Theorem]{A partial analog of integrability theorem for distributions on p-adic spaces and applications}
%
%
%
%
%
%
%
\begin{abstract}
Let $X$ be a smooth real algebraic variety. Let $\xi$ be a
distribution on it. One can define the singular support of $\xi$
to be the singular support of the $D_X$-module generated by $\xi$
(some times it is also called the characteristic variety). A powerful
property of the singular support is that it is a coisotropic
subvariety of $T^*X$. This is the integrability theorem (see
\cite{KKS, Mal, Gab}). This theorem turned out to be
useful in representation theory of real reductive groups (see e.g.
\cite{AG_AMOT, AS, Say}).

The aim of this paper is to give an analog of this theorem to the
non-Archimedean case. The theory of $D$-modules is not available
to us so we need a different definition of the singular support.
We use the notion  wave front set from \cite{Hef} and define the
singular support to be its Zariski closure. Then we prove that the
singular support satisfies some property that we call weakly
coisotropic, which is weaker than being coisotropic but is enough for
some applications. We also prove some other properties of the
singular support that were trivial in the Archimedean case (using
the algebraic definition) but not obvious in the non-Archimedean
case.

We provide two applications of those results:
\begin{itemize}
\item a non-Archimedean analog of the results of \cite{Say} concerning Gelfand property of nice symmetric pairs
\item  a proof of Multiplicity one Theorems for $GL_n$ which is uniform for all local fields.
This theorem was proven for the non-Archimedean case in
\cite{AGRS} and for the Archimedean case in \cite{AG_AMOT} and
\cite{SZ}.
\end{itemize}
\end{abstract}
\maketitle
\setcounter{tocdepth}{3}
\tableofcontents
\section{Introduction}
The theory of invariant distributions is widely used in
representation theory of reductive algebraic groups over local
fields. We can roughly divide this theory into two parts.
\begin{itemize}
\item Archimedean - distributions on smooth manifolds, Nash manifolds, real analytic manifolds, real algebraic manifolds, etc.
\item Non-Archimedean - distributions on l-spaces, p-adic analytic manifolds, p-adic algebraic manifolds, etc.
\end{itemize}
In general the non-Archimedean case of the theory of invariant
distributions is easier than the Archimedean one, but there is
one significant tool that is available only in the Archimedean
case. This tool is the theory of differential operators. One of
the powerful tools coming from the use of differential operators
is the notion of singular support (sometimes it is also called the
characteristic variety). The singular support of a distribution
$\xi$ on a real algebraic manifold $X$ is a subvariety of $T^*X$.
A deep and important property of the singular support is the fact
that it is coisotropic. This fact is the integrability theorem
(see \cite{KKS, Mal, Gab}). This theorem turned out
to be useful in the representation theory of real reductive groups
(see e.g. \cite{AG_AMOT,AS,Say}).

The aim of this paper is to give an analog of this theorem to the
non-Archimedean case. Though we didn't achieve a full analog of
the integrability theorem, we managed to formulate and prove some
partial analog of it. Namely we prove that the singular support
satisfies some property that we call weakly coisotropic, which is
weaker than being coisotropic but enough for some applications. We
also prove some other properties of the singular support that were
trivial in the Archimedean case
but not obvious in the non-Archimedean case.

We provide two applications of those results.

\begin{itemize}
\item We give a non-Archimedean analog of the results of \cite{Say} concerning Gelfand property of nice symmetric pairs.
\item We give a proof of Multiplicity one Theorems for $GL_n$ which is uniform for all local fields. This theorem was proven for the non-Archimedean case in \cite{AGRS} and for the non-Archimedean case in \cite{AG_AMOT} and \cite{SZ}.
\end{itemize}

\subsection{The singular support and the wave front set}$ $\\
The theory of $D$-modules is not available to us so we need a
different definition of singular support. We use the notion of
wave front set from \cite{Hef} and define the singular support to
be its Zariski closure. Unlike the algebraic definition of the
singular support, the definition of the wave front set is analytic
and uses Fourier transform instead of differential operators, this
is what makes it available for the non-Archimedean case.

Surprisingly, the fact that in the non-Archimedean case the
singular support is weakly coisotropic quite easily follows from
the basic properties of the wave front set developed in
\cite{Hef}. However another important property of the the singular
support that was trivial in the Archimedean case is not obvious in
the non-Archimedean case. Namely in presence of a group action one
can exhibit some restriction on the singular support of invariant
distribution. We also provide a non-Archimedean analog of this
property.

In general our results are based on the work \cite{Hef} where the
theory of the wave front set is developed for the non-Archimedean
case.
\subsection{Structure of the paper}$ $\\
In section \ref{prel} we give notations that will be used
throughout the paper and give some preliminaries on distributions,
including some results from \cite{Hef} on the wave front set.

In section \ref{coisotropic} we introduce the notion of coistropic
variety and weakly coistropic variety and discuss some properties
of them.

In section \ref{WFSS} we prove the main results on singular
support and the wave front set. We sum up the properties of
singular support in subsection \ref{SS}. In subsection
\ref{nondist} we apply those properties to get some technical
results that will be useful for proving Gelfand property.

In section \ref{GelSym} we generalize the results of \cite{Say} to
arbitrary local fields of characteristic 0.

In subsection \ref{GelSym_prel} we give the necessary
preliminaries for section \ref{GelSym}. In subsubsection \ref{Gel}
we provide basic preliminaries on Gelfand pairs. In subsubsection
\ref{Tame} we review a technique from \cite{AG_HC} for proving
that a given pair is a Gelfand pair. In subsubsections
\ref{SymPar}-\ref{def} we review a technique from \cite{AG_HC} and
\cite{AG_Reg} for proving that a given symmetric pair is a Gelfand
pair.

In section \ref{MOT} we indicate a proof of Multiplicity one
Theorems for $GL_n$ which is uniform for all local fields of
characteristic 0. This theorem was proven for the non-Archimidian
case in \cite{AGRS} and for the non-Archimidian case in
\cite{AG_AMOT} and \cite{SZ}.

\subsection{Acknowledgements}$ $\\
I wish to thank {\bf Dmitry Gourevitch} and {\bf Eitan Sayag} for fruitful discussions.
Also I cordially thank {\bf Dmitry Gourevitch} for his careful proof reading.
\section{Notations and preliminaries} \label{prel}
\begin{itemize}
\item Throughout the paper $F$ is a local field of characteristic zero.
\item All the algebraic varieties, analytic varieties and algebraic groups that we will consider will be defined over $F$.
\item By a reductive group we mean an algebraic reductive group.
\item Let $E$ be an extension of $F$. Let $G$ be an algebraic group
defined over $F$. We denote by $G_{E/F}$ the canonical algebraic
group defined over $F$ such that $G_{E/F}(F)=G(E)$.
\item By $Sp_{2n}$ we mean the symplectic  group of $2n \times 2n$ matrixes.
\item The word manifold will always mean that the object is smooth (e.g. by algebraic manifold we mean smooth algebraic
variety).
\item For a group $G$ acting on a set $X$ and a point $x \in X$ we denote by $Gx$ or by $G(x)$ the orbit of $x$ and by $G_x$ the stabilizer of $x$. we also denote by $X^G$ the set of $G$ invariant elements and for an element $g \in G$ denote by $X^g$ the set of $g$ invariant elements
\item An action of a Lie algebra $\g$ on a (smooth, algebraic, etc) manifold $M$ is a Lie algebra homomorphism from $\g$ to the Lie algebra of vector fields on $M$.
Note that an action of a (Lie, algebraic, etc) group on $M$ defines an action of its Lie algebra on $M$.
\item For a Lie algebra $\g$ acting on $M$, an element $\alpha \in \g$ and a point $x \in M$ we denote by $\alpha(x) \in T_xM$ t
he value at point $x$ of the vector field corresponding to
$\alpha$. We denote by $\g x \subset T_xM$ or by $\g (x) \subset
T_xM$ the image of the map $\alpha \mapsto \alpha(x)$ and by $\g_x
\subset \g$ its kernel. We denote $M^{\g}:=\{x \in M|\g x=0 \}$
and $M^{\alpha}:=\{x \in M|\alpha(x)=0 \}$, analogously to the
group case.
\item For manifolds  $L \subset M$ we
denote by $N_L^M:=(T_M|_L)/T_L $ the normal bundle to $L$ in $M$.
\item Denote by $CN_L^M:=(N_L^M)^*$ the conormal  bundle.
\item For a point $y \in L$ we denote by $N_{L,y}^M$ the normal space to $L$ in $M$
at the point $y$ and by $CN_{L,y}^M$ the conormal space.
\item Let $M,N$ be (smooth, algebraic, etc) manifolds. Let $E$ be a bundle over $N$. Let $\phi: M \to N$ be a morphism.
We denote by $\phi ^*(E)$ to be the pullback of $E$.
\item Let $M,N$ be (smooth, algebraic, etc) manifolds. Let $S \subset(T^*(N))$. Let $\phi: M \to N$ be a morphism.
We denote $\phi ^*(S):=d(\phi)^*(S \times_N M)$.
\item Let $M,N$ be topological spaces. Let $E$ be a over $N$. Let $\phi: M \to N$ be a morphism.
We denote by $\phi ^*(E)$ to be the pullback of $E$.
\item Let $V$ be a linear space. For a point $x=(v,\phi) \in V \times V^*$ we denote $\widehat{x}=(\phi,-v) \in V^* \times V$, similarly for subset $X \subset V \times V^*$ we define $\widehat{X}$. for a (smooth, algebraic, etc) manifold  and a subset $X  \subset  T^*(M \times V)$ we denote $\widehat{X}_V  \subset  T^*(M \times V^*)$ in a similar
way.
\item Let $B$ be a non-degenerate bilinear form on $V$. This gives an identification between $V$ and $V^*$ and therefore, by the previous notation, maps $F_B: V \times V \to V \times V$ and $F_B: T^*M \times V \times V \to T^*M \times V \times V$. If there is bo ambiguity we will denote it by
$F_V$.

\end{itemize}
\subsection{Distributions}$ $\\
In this paper we will refer to distributions on algebraic
varieties over archimedean and non-archimedean fields. In the
non-archimedean case we mean the notion of distributions on
$l$-spaces from \cite{BZ}, that is linear functionals on the space
of locally constant compactly supported functions.

We will use the following notations.

\begin{notation}
Let $X$ be an $l$-space.
\begin{itemize}
\item Denote by $\Sc(X)$ the space
of Schwartz functions on $X$ (i.e. locally  constant compactly
supported functions) Denote $\Sc^* (X):= \Sc(X)^*$ to be the dual
space to $\Sc(X)$.

\item For any locally constant sheaf $E$ over $X$ we denote by
$\Sc(X,E)$ the space of compactly supported sections of $E$ and by $\Sc^*(X,E)$ its dual space.

\item For any finite dimensional complex vector space $V$ we denote
$\Sc(X,V):=\Sc(X,X \times V)$ and $\Sc^*(X,V):=\Sc^*(X,X\times
V)$, where $X \times V$ is a constant sheaf.

\item Let $Z \subset X$ be a closed
subset. We denote
$$\Sc^*_X(Z):= \{\xi \in \Sc^*(X)|\Supp(\xi)\subset Z\}.$$
For a locally closed subset $Y \subset X$ we denote
$\Sc^*_X(Y):=\Sc^*_{X\setminus (\overline{Y} \setminus Y)}(Y)$. In
the same way, for any locally constant sheaf $E$ on $X$ we define
$\Sc^*_X(Y,E)$.

\item Suppose that $X$ is an analytic variety over a non-Archimedean field $F$. Then we define $D_X$ to be the sheaf of locally constant measures on $X$ (i.e. measures that locally are restriction of Haar measure on $F^n$). We denote $\G(X):=\Sc^*(X,D_X)$ and $\G(X,E):=\Sc^*(X,D_X \otimes
E^*)$.

\item For an analytic map $\phi:X \to Y$ of analytic manifolds over non-Archimedean field we denote by $\phi^*:\G(Y) \to \G(X)$ the pullback, similarly we denote $\phi^*:\G(Y,E) \to \G(X,\phi^*(E))$ for any locally constant sheaf
$E$.
\end{itemize}
\end{notation}

In the Archimedean case we will use the theory of Schwartz
functions and distributions as developed in \cite{AG_Sch}. This
theory is developed for Nash manifolds. Nash manifolds are smooth
semi-algebraic manifolds but in the present work only smooth real
algebraic manifolds are considered. Therefore the reader can
safely replace the word {\it Nash} by {\it smooth real algebraic}.

Schwartz functions are functions that decay, together with all
their derivatives, faster than any polynomial. On $\R^n$ it is the
usual notion of Schwartz function. For precise definitions of
those notions we refer the reader to \cite{AG_Sch}. We will use
the following notations.

\begin{notation}
Let $X$ be a Nash manifold.

Denote by $\Sc(X)$ the space of Schwartz functions on $X$. Denote
by $\Sc^* (X):= \Sc(X)^*$ the dual space to $\Sc(X)$. We define
$D_X$ to be the bundle of densities on $X$ for any Nash bundle $E$
on $X$ we define $\Sc^*(X,E),\Sc^*_X(Y),\G(X),\phi^*, etc$
analogously to the non-Archimedean case.
\end{notation}
\subsubsection{Invariant distributions}
\begin{proposition} \label{strat}
Let an $l$-group $G$ act on $l$-space $X$.
Let $Z \subset X$ be a closed subset.

Let $Z =
\bigcup_{i=0}^l Z_i$ be a $G$-invariant stratification of
$Z$. Let $\chi$ be a character of $G$. Suppose that for any $0 \leq i \leq l$ we have $\Sc^*(Z_i)^{G,\chi}=0$. Then
$\Sc^*_X(Z)^{G,\chi}=0$.
\end{proposition}
This proposition immediately follows from \cite[section 1.2]{BZ}.

\begin{proposition} \label{Arch_strat}
Let a Nash group $G$ act on a Nash manifold $X$.
Let $Z \subset X$ be a closed subset.

Let $Z =
\bigcup_{i=0}^l Z_i$ be a Nash $G$-invariant stratification of
$Z$. Let $\chi$ be a character of $G$. Suppose that for any $k \in
\Z_{\geq 0}$ and $0 \leq i \leq l$ we have $\Sc^*(Z_i,Sym^k(CN_{Z_i}^X))^{G,\chi}=0$. Then
$\Sc^*_X(Z)^{G,\chi}=0$.
\end{proposition}

This proposition immediately follows from
\cite[Corollary 7.2.6]{AGS}.

\begin{theorem}[Frobenius reciprocity] \label{Frob} 
Let an $l$-group (respectively Nash group) $G$ act transitively on an $l$-space (respectively Nash manifold) $Z$. Let
$\varphi:X \to Z$ be a $G$-equivariant map. Let $z\in Z$. Let $X_z$ be the fiber of $z$. Let $\chi$
be a character of $G$. Then $\Sc^*(X)^{G,\chi}$ is canonically
isomorphic to $\Sc^*(X_z)^{G_z,\chi \cdot \Delta_G|_{G_z} \cdot
\Delta_{G_z}^{-1}}$ where $\Delta$ denotes the modular character.
\end{theorem}
For a proof see \cite[section 1.5]{Ber}  for the non-Archimedean
case and \cite[Theorem 2.3.8]{AG_HC} for the non-Archimedean case.
\subsubsection{Fourier transform}
$ $\\
From now till the end of the paper we fix an additive character
$\kappa$ of $F$. If $F$ is Archimedean we fix $\kappa$ to be
defined by $\kappa(x):=e^{2\pi i \re(x)}$.

\begin{notation}
Let $V$ be a vector space over $F$. For any distribution $\xi \in
\Sc^*(V)$ we define  $\widehat{\xi} \in \G(V^*)$ to be its Fourier
transform.

For a space X (an $l$-space or a Nash manifold depending on $F$), for any distribution $\xi \in \Sc^*(X \times V)$ we define  $\widehat{\xi}_V \in \G(X \times V^*)$ to be its partial Fourer transform

Let $B$ be a non-degenerate bilinear form on $V$. Then $B$
identifies $\G(V^*)$  with $\Sc^*(V)$. We denote by $\Fou_B:
\Sc^*(V) \to \Sc^*(V)$ and $\Fou_B:\Sc^*(M \times V) \to
\Sc^*(M\times V)$ the corresponding Fourer transforms.

If there is no ambiguity, we will write $\Fou_V$, and sometimes just $\Fou$, instead of $\Fou_B$.
\end{notation}

We will use the following trivial observation.

\begin{lemma}
Let $V$ be a finite dimensional vector space over $F$. Let a
Nash group $G$ act linearly on $V$. Let $B$ be a $G$-invariant
non-degenerate symmetric bilinear form on $V$.
Let $\xi \in \Sc^*(V)$ be a $G$-invariant distribution. Then $\Fou_B(\xi)$ is also
$G$-invariant.
\end{lemma}

\begin{notation}
Let $V$ be a vector space over $F$. Consider the homothety action
of $F^{\times}$ on $V$ by $\rho(\lambda)v:= \lambda^{-1}v$. It
gives rise to an action $\rho$ of $F^{\times}$ on $\Sc^*(V)$.

Also, for any $\lambda \in F^{\times}$ denote
$|\lambda|:=\frac{dx}{\rho(\lambda)dx}$, where $dx$ denotes the
Haar measure on $F$.
\end{notation}
\begin{notation}
Let $V$ be a vector space over $F$. Let $B$ be a non-degenerate
symmetric bilinear form on $V$. We denote $$Z(B):=\{x \in
V(F)|B(x,x)=0 \}.$$
\end{notation}

\begin{theorem} [Homogeneity Theorem] \label{ArchHom}
Let $V$ be a vector space over $F$. Let $B$ be a non-degenerate
symmetric bilinear form on $V$. Let $M$ be a space(an $l$-space or a Nash manifold depending on $F$). Let $L
\subset \Sc^*_{V(F)\times M}(Z(B)\times M)$ be a non-zero subspace
such that $\forall \xi \in L $ we have $\Fou_B(\xi) \in L$ and $B
\xi \in L$ (here $B$ is interpreted as a quadratic form).

Then there exist a non-zero distribution $\xi \in L$ and a unitary character $u$ of $F^{\times}$ such that either
$\rho(\lambda)\xi = || \lambda ||^{\frac{dimV}{2}} u(\lambda) \xi$ for any $\lambda \in F^{\times }$ or
$\rho(\lambda)\xi = | \lambda |^{\frac{dimV}{2}+1} u(\lambda) \xi$ for any $\lambda \in F^{\times }$.
\end{theorem}
For a proof see \cite[Theorem 5.1.7]{AG_HC}.

\subsubsection{The wave front set}$ $\\
In this subsubsection $F$ is a non-Archimedean field. We will use
the notion of the wave front set of a distribution on analytic
space from \cite{Hef}. First we will remind it for a distribution
on an open subset of $F^n$.

\begin{definition}
Let $U \subset F^n$ be an open subset and $\xi \in \Sc^*(U)$ be a
distribution. We say that $\xi$ is smooth at $(x_0, v_0) \in T^*U$
if there are open neighborhoods $A$ of $x_0$ and $B$ of $v_0$ such
that for any $\phi \in \Sc(A)$ there is an $N_\phi > 0$ for which
for any $\lambda \in F$ satisfying $\lambda > N_\phi$ we have
$\widehat{(\phi \xi)}|_{\lambda B} = 0$. The complement in $T^*U$
of the set of smooth pairs $(x_0, v_0)$ of $\xi$ is called the
wave front set of $\xi$ and denoted by $WF(\xi)$.
\end{definition}

\begin{remark}
This notion appears in \cite{Hef} with two differences.

1) The notion in \cite{Hef} is more general and depends on some
subgroup $\Lambda \subset F$, in our case $\Lambda = F$.

2) The notion in  \cite{Hef} defines the wave front set of $\xi$ to
be a subset in $T^*U -U \times 0$. In our notation this subset will be
$WF(\xi) -U \times 0$.
\end{remark}

The following lemmas are trivial

\begin{lemma}
Let $U \subset F^n$ be an open subset and $\xi \in \Sc^*(U)$ be a
distribution. Then $WF(\xi)$ is closed, invariant with respect to the
homothety $(x,v) \mapsto (x,\lambda v)$ and
$$p_U(WF(\xi))=WF(\xi) \cap (U \times 0) =\Supp(\xi).$$
\end{lemma}
\begin{lemma}
Let $V \subset U \subset F^n$ be open subsets and  $\xi \in
\Sc^*(U)$ then $WF(\xi|_V) = WF(\xi) \cap p_U^{-1}(V).$
\end{lemma}

\begin{lemma}
Let $U \subset F^n$ be an open subset, $\xi_1,\xi_2 \in \Sc^*(X)$
be distributions and $f_1,f_2$ be locally constant functions on
$X$. Then $WF(f_1 \xi_1+f_2 \xi_2) \subset WF(\xi_1) \cup
WF(\xi_2)$.
\end{lemma}

\begin{corollary}
For any locally constant sheaf $E$ on $U$ we can define the wave
front set of any element in $\Sc^*(U,E)$ and $\G(U,E)$.
\end{corollary}
We will use the following theorem from \cite{Hef}, see Theorem
2.8.

\begin{theorem} \label{submrtion}
Let $U \subset F^m$ and $V \subset F^n$ be open subsets, and
suppose that $f: U \to V$ is an analytic submersion. Then for any
$\xi \in \G(V)$ we have $WF(f^*(\xi)) \subset f^*(WF(\xi))$.
\end{theorem}

\begin{corollary} \label{iso}
Let $V, U \subset F^n$ be open subsets and $f: V \to U$ be an
analytic isomorphism. Then for any $\xi \in \G(V)$ we have
$WF(f^*(\xi)) = f^*(WF(\xi))$.
\end{corollary}

\begin{corollary}
Let $X$ be an analytic manifold, $E$ be a locally constant sheaf
on $X$. We can define the the wave front set of any element in
$\Sc^*(X,E)$ and $\G(X,E)$. Moreover, Theorem \ref{submrtion}
holds for submersions between analytic manifolds.
\end{corollary}

\section{Coisotropic varieties} \label{coisotropic}
\setcounter{lemma}{0}
\begin{definition}
Let $M$ be a smooth algebraic variety and $\omega$ be a symplectic
form on it.
Let $Z\subset M$ be an algebraic subvariety. We call it {\bf $M$-coisotropic} if one of the following equivalent conditions holds.\\
(i) The ideal sheaf of regular functions that vanish on $\overline{Z}$ is closed under Poisson bracket. \\
(ii) At every smooth point $z \in Z$ we have  $T_zZ \supset (T_zZ)^{\bot}$. Here, $(T_zZ)^{\bot}$ denotes the orthogonal complement
to $T_zZ$ in $T_zM$ with respect to $\omega$. \\
(iii) For a generic smooth point $z \in Z$ we have $T_zZ \supset
(T_zZ)^{\bot}$.

If there is no ambiguity, we will call $Z$ a coisotropic variety.
\end{definition}
Note that every non-empty $M$-coisotropic variety is of dimension
at least $\frac{1}{2}\dim M$.

\begin{notation}
For a smooth algebraic variety $X$ we always consider the standard
symplectic form on $T^*X$. Also, we denote by $p_X:T^*X \to X$ the
standard projection.
\end{notation}

\begin{definition}
Let $(V,\omega)$ be a symplectic vector space with a fixed Lagrangian subspace $L \subset V$. Let $p: V \to V/L$ be the standard projection. Let $Z \subset V$ be a linear subspace. We call it $V$-weakly coisotropic with respect to $L$ if one of the following equivalent conditions holds.\\
(i) $p(Z) \supset p(Z^ \bot)$. Here, $Z^{\bot}$ denotes the orthogonal complement with respect to $\omega.$\\
(ii) $p(Z)^\bot \subset Z \cap L$. Here, $p(Z)^\bot$ denotes the orthogonal complement in $L$ under the identification $L \cong (V/L)^*$.
\end{definition}

\begin{definition}
Let $X$ be a smooth algebraic variety. Let $Z\subset T^*X$ be an algebraic subvariety. We call it {\bf $T^*X$-weakly coisotropic} if one of the following equivalent conditions holds.\\
(i)  At every smooth point $z \in Z$ the space $T_z(Z)$ is $T_z(T^*(X))$ -weakly coisotropic with respect to $Ker(d p_X).$\\
(ii)For a generic smooth point $z \in Z$ the space $T_z(Z)$ is $T_z(T^*(X))$ -weakly coisotropic with respect to $Ker(d p_X).$\\
(iii) For a generic smooth point  $x \in Z$ and for a generic smooth point  $y \in p_X^{-1}(x) \cap Z$ we have
$CN_{p_X(Z),x}^X \subset T_y(p_X^{-1}(x) \cap Z).$\\
(iv) For any smooth point  $x \in p_X(Z)$ the fiber $p_X^{-1}(x) \cap Z$ is locally invariant with respect to shifts by $CN_{p_X(Z),x}^X$ i.e. for any point $y \in p_X^{-1}(x)$ the intersection $(y+CN_{p_X(Z),x}^X) \cap (p_X^{-1}(x)\cap Z)$ is Zariski open in $y+CN_{p_X(Z)}$.

If there is no ambiguity, we will call $Z$ a weakly coisotropic variety.
\end{definition}
Note that every non-empty $T^*X$-weakly coisotropic variety is of dimension
at least $\dim X$.

The following lemma is straightforward.
\begin{lemma}
Any $T^*X$-coisotropic variety is $T^*X$-weakly coisotropic.
\end{lemma}
\begin{proposition}
Let $X$ be a smooth algebraic variety with a symplectic form on
it. Let $R \subset T^*X$ be an algebraic subvariety. Then there
exists a maximal $T^*X$-weakly coisotropic subvariety of $R$ i.e. a
$T^*X$-weakly coisotropic subvariety $T \subset M$ that includes all
$T^*X$-weakly coisotropic subvarieties of $R$.
\end{proposition}
\begin{proof}
Let $T'$ be the union of all smooth $T^*X$-weakly coisotropic
subvarieties of $R$. Let $T$ be the Zariski closure of $T'$ in
$R$. It is easy to see that $T$ is the maximal $T^*X$-weakly
coisotropic subvariety of $R$.
\end{proof}
The following lemma is trivial.
\begin{lemma}
Let $X$ be a smooth algebraic variety. Let a group $G$ act on $X$ this induces an action on $T^*X$. Let
$S \subset T^* X$ be a $G$-invariant subvariety. Then the maximal $T^*X$-weakly coisotropic subvariety of $S$ is also $G$-invariant.
\end{lemma}
\begin{notation}
Let $Y$ be a smooth algebraic variety. Let $Z \subset Y$ be a
smooth subvariety and $R \subset T^*Y$  be any subvariety. We
define {\bf the restriction $R|_Z \subset T^*Z$ of $R$ to $Z$} by
$R|_Z:=i^*(R)$, where $i:Z \to Y$ is the embedding.
\end{notation}
\begin{lemma} \label{Restriction}
Let $Y$ be a smooth algebraic variety. Let $Z \subset Y$ be a
smooth subvariety and $R \subset T^*Y$  be a weakly coisotropic
subvariety.
Assume that any smooth point $z \in Z \cap p_Y(R)$ is also a smooth point of $p_Y(R)$ and we have $T_z(Z \cap p_Y(R)) = T_z(Z) \cap T_z (p_Y(R))$.

Then $R|_Z$ is $T^*Z$-weakly coisotropic.
\end{lemma}
\begin{proof}
Let $x \in Z$, let $M:= p_Y^{-1}(x) \cap R \subset p_Y^{-1}(x)$ and $L:=CN_{p_Y(R),x}^Y \subset p_Y^{-1}(x).$ We know that $M$ is locally invariant with respect to shifts in $L$. Let $M':= p_Z^{-1}(x) \cap R|_Z \subset p_Z^{-1}(x)$ and $L':=CN_{p_Z(R|_Z),x}^Y \subset p_Z^{-1}(x).$ We want to show that $M'$ is locally invariant with respect to shifts in $L'$. Let $q: p_Y^{-1}(x) \to p_Z^{-1}(x)$ be the standard projection. Note that $M'=q(M)$ and $L'=q(L)$. Now clearly $M'$ is locally invariant with respect to shifts in $L'$.
\end{proof}

\begin{corollary} \label{PreGeoFrob}
Let $Y$ be a smooth algebraic variety.  Let an algebraic group $H$
act on $Y$. Let $q:Y \to B$ be an $H$-equivariant morphism. Let $O
\subset B$ be an orbit. Consider the natural action of $G$ on
$T^*Y$ and let $R \subset T^*Y$ be an $H$-invariant subvariety.
Suppose that $p_Y(R) \subset q^{-1}(O)$. Let $x \in O$. Denote
$Y_x:= q^{-1}(x)$. Then

\itemize{ \item if $R$ is $T^*Y$-weakly coisotropic then $R|_{Y_x}$ is $T^*(Y_x)$-weakly coisotropic.}
\end{corollary}
\begin{corollary} \label{GeoFrob}
In the notation of the previous corollary, if $R|_{Y_x}$ has no
(non-empty) $T^*(Y_x)$-weakly coisotropic subvarieties then $R$ has no
(non-empty) $T^*(Y)$-weakly coisotropic subvarieties.
\end{corollary}

\begin{remark}
The  results  on weakly coistropic varieties  that we presented
here have versions for coistropic varieties, see \cite[section 5.1]{AG_AMOT}.
\end{remark}

\section{Properties of singular support and the wave front set} \label{WFSS}
\subsection{The wave front set}$ $\\
In this subsection $F$ is a non-Archimedean field.
\begin{theorem}
Let $Y \subset X$ be algebraic varieties, let $y \in Y(F)$ and
suppose that $X$ is smooth and $Y$ is smooth at $y$. Let $\xi \in
\Sc^*(X(F),E)$ and suppose that $\Supp(\xi) \subset Y(F).$ Then
$WF(\xi) \cap p_X^{-1}(y)(F)$ is invariant with respect to shifts
by $CN_{Y,y}^X(F).$
\end{theorem}
This theorem immediately follows from the following one
\begin{theorem} \label{W_Gab_ann}
Let $Y \subset X$ be analytic manifolds and let $y \in Y$. Let $\xi
\in \Sc_X^*(Y)$ and suppose that $\Supp(\xi) \subset Y.$ Then
$WF(\xi) \cap p_X^{-1}(y)$ is invariant with respect to shifts by
$CN_{Y,y}^X$.
\end{theorem}
In order to prove this theorem we will need the following standard
lemma which is a version of the implicit function theorem.
\begin{lemma}
Let $Y \subset X$ be analytic manifolds. Let $n:=\dim(X)$ and
$k:=\dim(Y)$. Let $y \in Y.$ Then there exist a open neighborhood
$y \in U \subset X$ and an analytic isomophism $\phi : U \to W$,
where $W$ is open subset of $F^n$  such that $\phi(Y \cap U) = W
\cap F^k$, where $F^k \subset F^n$ is a coordinate subspace.
\end{lemma}
\begin{proof}[Proof of theorem \ref{W_Gab_ann}]
$ $

Case 1: $X = F^n$, $Y = F^k$.\\
in this case the theorem follows from the fact that if a
distribution on $F^n$ is supported on $F^k$ then its Fourier
transform is  invariant with respect to shifts by the orthogonal
complement to $F^k$.

Case 2: $X = U \subset F^n$, $Y = F^k \cap U$, where $U \subset F^n$ is open.\\
Follows immediately from the previous case.

Case 3: the general case.\\
Follows from the previous case using the lemma and theorem
\ref{iso}.
\end{proof}
\begin{theorem}
Let an algebraic group $G$ act on a smooth algebraic variety $X$.
Let $\g$ be the  Lie algebra of $G$. Let $\xi \in \Sc^*(X)^G$.
Then $WF(\xi) \subset \{(x,v) \in T^*X(F)|v(\g x)=0 \}.$
\end{theorem}
We will prove a slightly more general theorem.
\begin{theorem}\label{G_inv_dist}
Let an analytic group $G$ act on an analytic manifold $X$. Let $E$
be a $G$-equivariant locally constant sheaf on $X$. Let $\xi \in
\G(X,E)^G$. Then $WF(\xi) \subset \{(x,v) \in T^*X(F)|v(\g x)=0
\}.$
\end{theorem}
In order to prove this theorem we will need the following easy
lemma.
\begin{lemma}
Let $X,Y$ be analytic manifolds. Let $E$ be a locally constant
sheaf on $X$. Let $\xi \in \G(X,E)$. Let $p : X \times Y  \to X$
be the projection. Then $WF(p^*(\xi)) = p^*(WF(\xi)).$
\end{lemma}

\begin{proof}[Proof of theorem \ref{G_inv_dist}]
Consider the action map $m: G \times X \to X$ and the projection
$p: G \times X \to X$. Let $S := WF(\xi)$. We are given an
isomorphism $p^*(E) \cong m^*(E)$ and we know that under this
identification $p^*(\xi)=m^*(\xi)$. Therefore
$WF(p^*(\xi))=WF(m^*(\xi))$. By the lemma we have $WF(p^*(\xi))=
p^*(S)$. by theorem \ref{submrtion} we have $WF(m^*(\xi)) \subset
m^*(S)$. Thus we got $p^*(S) \subset m^*(S)$ which implies the
requested inclusion.
\end{proof}
\subsection{Singular support} \label{SS}
\begin{definition}
Let $X$ be a smooth algebraic variety let $\xi \in \Sc^*(X(F))$.
We will now define the singular support of $\xi$, it is an
algebraic subvariety of $T^*X$ and we will denote it by
$SS(\xi)$.

In the case when $F$ is non-Archimedean we define it to be the
Zariski closure of $WF(\xi)$. In the case when $F$ is Archimedean
we define it to be the singular support of the $D_X$-module
generated by $\xi$ (as in \cite{AG_AMOT}).
\end{definition}

In \cite[section 2.3]{AG_AMOT} the following list of
properties of the singular support for the Archimedean case was introduced:

Let $X$ be a smooth algebraic variety.
\begin{enumerate}
\item \label{Supp2SS}
Let $\xi \in \Sc^*(X (F))$.  Then $\overline{\Supp(\xi)}_{Zar} =
p_X(SS(\xi))(F)$, where $\overline{\Supp(\xi)}_{Zar} $ denotes
the Zariski  closure of
 $\Supp(\xi)$.
\item \label{Ginv}
Let an algebraic group $G$ act on $X$. Let $\g$ denote the
Lie algebra of $G$. Let $\xi \in \Sc^*(X(F))^{G(F)}$. Then
$$SS(\xi) \subset \{(x,\phi) \in T^*X \, | \, \forall \alpha \in
\g \, \phi(\alpha(x)) =0\}.$$
\item \label{Fou}
Let $V$ be a linear space. Let $Z \subset X \times V$ be a closed
subvariety, invariant with respect to homotheties in $V$. Suppose
that $\Supp(\xi) \subset Z(F)$. Then $SS(\Fou_V(\xi)) \subset
F_V(p_{X \times V}^{-1}(Z))$.
\item \label{Gab}
Let $X$ be a smooth algebraic
variety. Let $\xi \in \Sc^*(X(F))$. Then $SS(\xi)$ is
coisotropic.
\end{enumerate}
\begin{remark}
Property \ref{Gab} is a corollary of the integrability theorem
(see \cite{KKS, Mal,Gab}).
\end{remark}
The result of the last subsection shows that those properties are
satisfied for the non-Archimedean case with the following
modification, property \ref{Gab} should be replaced by the
following weaker one:

(4') $\,$ Let $X$ be a smooth algebraic variety. Let $\xi \in \Sc^*(X(F))$. Then $SS(\xi)$ is
weakly coisotropic.

We conjecture that property \ref{Gab} holds for the
non-Archimedean case without modification.
\subsection{Distributions on non distinguished nilpotent orbits} \label{nondist}
$ $\\
In this subsection we deduce from the properties of singular
support some technical results that are useful for proving Gelfand
property.

\begin{notation}
Let $V$ be an algebraic finite dimensional representation over $F$
of a reductive group $G$. We denote $$Q(V):=(V/V^G)(F).$$ Since $G$ is reductive, there is a canonical embedding $Q(V) \hookrightarrow V(F)$. We also denote
$$\Gamma(V) =\{y \in V(F) \, | \, \overline{G(F)}y \ni 0\}.$$
Note that $\Gamma(V)\subset Q(V)$. We denote also $R(V):= Q(V)-\Gamma(V).$
\end{notation}

\begin{definition}
Let $V$ be an algebraic finite dimensional representation over $F$
of a reductive group $G$. Suppose that there is a finite number of
$G$ orbits in $\Gamma(V).$ Let $x \in \Gamma(V).$ We will call it
$G$-distinguished, if $CN_{G x, x}^{Q(V)} \subset \Gamma(V^*)$. We
will call a $G$ orbit $G$-distinguished if all (or equivalently
one of) its elements are $G$- distinguished.

If there is no ambiguity we will omit the "$G$-".
\end{definition}

\begin{example} \label{group_case}
For the case of a semi-simple group acting on its Lie algebra, the
notion of $G$-distinguished element coincides with the standard
notion of distinguished nilpotent element. In particular, in the case when $G=SL_n$ and $V = sl_n$ the set of $G$-distinguished elements is exactly the set of regular nilpotent elements.
\end{example}

\begin{proposition} \label{non_disting_no_cois}
Let $V$ be an algebraic finite dimensional representation over $F$
of a reductive group $G$. Suppose that there is a finite number of
$G$ orbits on $\Gamma(V).$ Let $W:=Q(V)$, let $A$ be the set of
non-distinguished elements in $\Gamma(V).$ Then there are no non-empty
$W \times W^*$-weakly coisotropic subvarieties of $A \times
\Gamma(V^*).$
\end{proposition}
The proof is clear.
\begin{corollary} \label{non_disting_no_dists}
Let $\xi \in \Sc^*(W)$ and suppose that  $\Supp(\xi) \subset
\Gamma(V)$ and $supp(\widehat{\xi}) \subset  \Gamma(V^*)$. Then the
set of distinguished elements in $\Supp(\xi)$ is dense in
$\Supp(\xi)$
\end{corollary}
\begin{remark}
In the same way one can prove an analogus result for distributions on $W \times M$ for any analitic manifold $M$.
\end{remark}

\section{Applications towards Gelfand properties of symmetric pairs} \label{GelSym}
\setcounter{lemma}{0} In this section we will use the property of
singular support to generate the results of \cite{Say} for any
local field of characteristic $0$. Namely we prove that a big
class of symmetric pairs are {\it regular}. The property of
regularity of symmetric pair was introduced in \cite{AG_HC} and
was shown to be useful for proving Gelfand property. We will give
more details on the regularity property and its connections with
Gelfand property in subsubsections \ref{SymPar}-\ref{def}.
\subsection{Preliminaries} \label{GelSym_prel}$ $\\
In this subsection we give the necessary
preliminaries for section \ref{GelSym}.
\subsubsection{Gelfand pairs} \label{Gel}
$ $\\
In this subsubsection we recall a technique due to Gelfand and Kazhdan
(see \cite{GK}) which allows to deduce statements in representation
theory from statements on invariant distributions. For more
detailed description see \cite[section 2]{AGS}.

\begin{definition}
Let $G$ be a reductive group. By an \textbf{admissible
representation of} $G$ we mean an admissible representation of
$G(F)$ if $F$ is non-Archimedean (see \cite{BZ}) and admissible
smooth \Fre representation of $G(F)$ if $F$ is Archimedean.
\end{definition}

We now introduce three notions of Gelfand pair.

\begin{definition}\label{GPs}
Let $H \subset G$ be a pair of reductive groups.
\begin{itemize}
\item We say that $(G,H)$ satisfy {\bf GP1} if for any irreducible
admissible representation $(\pi,E)$ of $G$ we have
$$\dim Hom_{H(F)}(E,\cc) \leq 1$$

\item We say that $(G,H)$ satisfy {\bf GP2} if for any irreducible
admissible representation $(\pi,E)$ of $G$ we have
$$\dim Hom_{H(F)}(E,\cc) \cdot \dim Hom_{H(F)}(\widetilde{E},\cc)\leq
1$$

\item We say that $(G,H)$ satisfy {\bf GP3} if for any irreducible
{\bf unitary} representation $(\pi,\mathcal{H})$ of $G(F)$ on a
Hilbert space $\mathcal{H}$ we have
$$\dim Hom_{H(F)}(\mathcal{H}^{\infty},\cc) \leq 1.$$
\end{itemize}

\end{definition}
Property GP1 was established by Gelfand and Kazhdan in certain
$p$-adic cases (see \cite{GK}). Property GP2 was introduced in
\cite{Gross} in the $p$-adic setting. Property GP3 was studied
extensively by various authors under the name {\bf generalized
Gelfand pair} both in the real and $p$-adic settings (see e.g.
\cite{vD,BvD}).

We have the following straightforward proposition.

\begin{proposition}
$GP1 \Rightarrow GP2 \Rightarrow GP3.$
\end{proposition}

We will use the following theorem from \cite{AGS} which is a
version of a classical theorem of Gelfand and Kazhdan.

\begin{theorem}\label{DistCrit}
Let $H \subset G$ be reductive groups and let $\tau$ be an
involutive anti-automorphism of $G$ and assume that $\tau(H)=H$.
Suppose $\tau(\xi)=\xi$ for all bi $H(F)$-invariant distributions
$\xi$ on $G(F)$. Then $(G,H)$ satisfies GP2.
\end{theorem}

\begin{remark}
In many cases it terns out that GP2 is equivalent to GP1.
\end{remark}

\subsubsection{Tame actions} \label{Tame}
$ $\\
In this subsubsection we review some tools developed in \cite{AG_HC} for solving problems of the following type. A reductive group $G$ acts on a smooth affine variety $X$, and $\tau$ is an automorphism of $X$ which normalizes the action of $G$. We want to check whether any $G(F)$-invariant Schwartz distribution on $X(F)$ is also $\tau$-invariant.

\begin{definition}
Let $\pi$ be an action of a reductive group $G$ on a smooth affine variety $X$.
We say that an algebraic automorphism $\tau$ of $X$ is \textbf{$G$-admissible} if \\
(i) $\pi(G(F))$ is of index $\leq 2$ in the group of automorphisms
of $X$
generated by $\pi(G(F))$ and $\tau$.\\
(ii) For any closed $G(F)$ orbit $O \subset X(F)$, we have
$\tau(O)=O$.
\end{definition}
\begin{definition}
We call an action of a reductive group $G$ on a smooth affine
variety $X$ \textbf{tame} if for any $G$-admissible $\tau : X \to
X$, we have $\Sc^*(X(F))^{G(F)} \subset \Sc^*(X(F))^{\tau}.$
\end{definition}

\begin{definition}
We call an algebraic representation  of a reductive group $G$ on a
finite dimensional linear space $V$ over $F$ \textbf{linearly
tame} if for any $G$-admissible linear map $\tau : V \to V$, we
have $\Sc^*(V(F))^{G(F)} \subset \Sc^*(V(F))^{\tau}.$

We call a representation \textbf{weakly linearly tame} if for any
$G$-admissible linear map $\tau : V \to V$, such that
$\Sc^*(R(V))^{G(F)} \subset \Sc^*(R(V))^{\tau}$ we have
$\Sc^*(Q(V))^{G(F)} \subset \Sc^*(Q(V))^{\tau}.$
\end{definition}

\begin{theorem} \label{Invol_HC}
Let a reductive group $G$ act on a smooth affine variety $X$.
Suppose that for any $G$-semisimple $x \in X(F)$, the action of
$G_x$ on $N_{Gx,x}^X$ is weakly linearly tame. Then the action of
$G$ on $X$ is tame.
\end{theorem}
For a proof see \cite[Theorem 6.0.5]{AG_HC}.
\begin{definition}
We call an algebraic representation  of a reductive group $G$ on a
finite dimensional linear space $V$ over $F$ \textbf{special} if
for any $\xi \in \Sc^*_{Q(V)}(\Gamma(V))^{G(F)}$ such that for any
$G$-invariant decomposition $Q(V) = W_1 \oplus W_2$ and any two
$G$-invariant symmetric non-degenerate  bilinear forms $B_i$ on
$W_i$ the Fourier transforms $\Fou_{B_i}(\xi)$ are also supported
in $\Gamma(V)$, we have $\xi = 0$.
\end{definition}

\begin{proposition} \label{SpecWeakTameAct}
Every special algebraic representation $V$ of a reductive group
$G$ is weakly linearly tame.
\end{proposition}
For a proof see \cite[Proposition 6.0.7]{AG_HC}.

\subsubsection{Symmetric pairs} \label{SymPar}
$ $\\
In the coming 4 subsubsections we review some tools developed in \cite{AG_HC} 
that enable to prove that a symmetric pair is a Gelfand pair.
\begin{definition}
A \textbf{symmetric pair} is a triple $(G,H,\theta)$ where $H
\subset G$ are reductive groups, and $\theta$ is an involution of
$G$ such that $H = G^{\theta}$. We call a symmetric pair
\textbf{connected} if $G/H$ is connected.

For a symmetric pair $(G,H,\theta)$ we define an anti-involution
$\sigma :G \to G$ by $\sigma(g):=\theta(g^{-1})$, denote $\g:=Lie
G$, $\h := LieH$, $\gd:=\{a \in \g | \theta(a)=-a\}$. Note that
$H$ acts on $\gd$ by the adjoint action. Denote also
$G^{\sigma}:=\{g \in G| \sigma(g)=g\}$ and define a
\textbf{symmetrization map} $s:G \to G^{\sigma}$ by $s(g):=g
\sigma(g)$.

In case when the involution is obvious we will omit it.
\end{definition}

\begin{remark}
Let $(G,H,\theta)$ be a symmetric pair. Then $\g$ has a $\Z/2\Z$
grading given by $\theta$.
\end{remark}

\begin{definition}
Let $(G_1,H_1,\theta_1)$ and $(G_2,H_2,\theta_2)$ be symmetric
pairs. We define their \textbf{product} to be the symmetric pair
$(G_1 \times G_2,H_1 \times H_2,\theta_1 \times \theta_2)$.
\end{definition}

\begin{definition}
We call a symmetric pair $(G,H,\theta)$ \textbf{good} if for any
closed $H(F) \times H(F)$ orbit $O \subset G(F)$, we have
$\sigma(O)=O$.
\end{definition}

\begin{proposition} \label{GoodCrit}
Every connected symmetric pair over $\C$ is good.
\end{proposition}
For a proof see e.g. \cite[Corollary 7.1.7]{AG_HC}.

\begin{definition}
We say that a symmetric pair $(G,H,\theta)$ is a \textbf{GK pair}
if any $H(F) \times H(F)$-invariant distribution on $G(F)$ is
$\sigma$-invariant.
\end{definition}

\begin{remark}
Theorem \ref{DistCrit} implies that any GK pair satisfies GP2.
\end{remark}
\subsubsection{Descendants of symmetric pairs}
\begin{proposition} \label{PropDescend}
Let  $(G,H,\theta)$ be a symmetric pair. Let $g \in G(F)$ such
that $HgH$ is closed. Let $x=s(g)$. Then $x$ is a semisimple
element of $G$.
\end{proposition}
For a proof see e.g. \cite[Proposition 7.2.1]{AG_HC}.
\begin{definition}
In the notations of the previous proposition we will say that the
pair $(G_x,H_x,\theta|_{G_x})$ is a \textbf{descendant} of
$(G,H,\theta)$.
\end{definition}

\subsubsection{Tame symmetric pairs}
\begin{definition}$ $\\
\begin{itemize}
\item We call a symmetric pair $(G,H,\theta)$ \textbf{tame} if the action of  $H \times$ on $G$ is tame
\item We call a symmetric pair $(G,H,\theta)$ \textbf{linearly tame} if the action of  $H$ on $\gd$ is linearly tame
\item We call a symmetric pair $(G,H,\theta)$ \textbf{weakly linearly tame} if the action of  $H$ on $\gd$ is weakly linearly tame
\item We call a symmetric pair $(G,H,\theta)$ \textbf{special} if the action of  $H$ on $\gd$ is special
\end{itemize}
\end{definition}

\begin{remark}
Evidently, any good tame symmetric pair is a GK pair.
\end{remark}

\begin{theorem} \label{LinDes}
Let $(G,H,\theta)$ be a symmetric pair. Suppose that all its
descendants (including itself)  are weakly linearly tame. Then
$(G,H,\theta)$ is tame.
\end{theorem}

For a proof see \cite[Theorem 7.3.3]{AG_HC}.
\subsubsection{Regular symmetric pairs}
\begin{definition}
Let $(G,H,\theta)$ be a symmetric pair. We call an element $g \in
G(F)$ \textbf{admissible} if\\
(i) $Ad(g)$ commutes with $\theta$ (or, equivalently, $s(g)\in Z(G)$) and \\
(ii) $Ad(g)|_{\g^{\sigma}}$ is $H$-admissible.
\end{definition}

\begin{definition}
We call a symmetric pair $(G,H,\theta)$ \textbf{regular} if for
any admissible $g \in G(F)$ such that every $H(F)$-invariant
distribution on $R_{G,H}$ is also $Ad(g)$-invariant,
we have\\
(*) every $H(F)$-invariant distribution on $Q(\gd)$ is also
$Ad(g)$-invariant.
\end{definition}

The following two propositions are evident.
\begin{proposition} \label{TrivReg}
Let $(G,H,\theta)$ be symmetric pair. Suppose that any $g \in
G(F)$ satisfying $\sigma(g)g \in Z(G(F))$ lies in $Z(G(F))H(F)$.
Then $(G,H,\theta)$ is regular. In particular if the normalizer of
$H(F)$ lies inside $Z(G(F))H(F)$ then $(G,H,\theta)$ is regular.
\end{proposition}

\begin{proposition} $ $\\
(i) Any weakly linearly tame pair is regular. \\
(ii) A product of regular pairs is regular (see \cite[Proposition 7.4.4]{AG_HC}).
\end{proposition}

The importance of the notion of regular pair is demonstrated by
the following theorem.

\begin{theorem} \label{GoodHerRegGK}
Let $(G,H,\theta)$ be a good symmetric pair such that all its
descendants (including itself) are regular. Then it is a GK pair.
\end{theorem}
For a proof see \cite[Theorem 7.4.5]{AG_HC}.
\subsubsection{Defects of symmetric pairs} \label{def}
$ $\\
In this subsection we review some tools developed in \cite{AG_HC} and \cite{AG_Reg}
that enable to prove that a symmetric pair is special.

\begin{definition}
We fix standard basis $e,h,f$ of $sl_2(F)$. We fix a grading on
$sl_2(F)$ given by $h \in sl_2(F)_0$ and $e,f \in sl_2(F)_1$. A
\textbf{graded representation of $sl_2$} is a representation of
$sl_2$ on a graded vector space $V=V_0 \oplus V_1$ such that
$sl_2(F)_i(V_j) \subset V_{i+j}$ where $i,j \in \Z/2\Z$.
\end{definition}

The following lemma is standard.
\begin{lemma}$ $\\
(i) Every graded representation of $sl_2$ which is
irreducible as a graded representation is irreducible just as a representation.\\
(ii) Every irreducible representation $V$ of $sl_2$ admits exactly
two gradings. In one highest weight vector lies in $V_0$ and in
the other in $V_1$.
\end{lemma}

\begin{definition}
We denote by $V_{\lambda}^{w}$ the irreducible graded
representation of $sl_2$ with highest weight $\lambda$ and highest
weight vector of parity $p$ where $w = (-1)^p$.
\end{definition}

The following lemma is straightforward.

\begin{lemma} \label{Graded_star}
$(V_{\lambda}^w)^* = V_{\lambda}^{w(-1)^{\lambda}}$.
\end{lemma}
\begin{definition}
Let $\pi$ be a graded representation of $sl_2$. We define the
\textbf{defect} of $\pi$ to be 
$$\de(\pi)=\tr(h|_{(\pi^e)_0})-\dim(\pi_1).$$
\end{definition}
The following lemma is straightforward

\begin{lemma} \label{Defects}
\begin{align}
&\de(\pi \oplus \tau)=\de(\pi) + \de(\tau)\\
&\de(V_{\lambda}^w) =\frac{1}{2}(\lambda w + w( \frac{1
+(-1)^{\lambda}}{2})-1)= \frac{1}{2} \left\{%
\begin{array}{ll}
   \lambda w + w-1 & \lambda \text{ is even} \\
    \lambda w-1 & \lambda \text{ is odd}
\end{array}%
\right.
\end{align}

\end{lemma}

\begin{lemma}
Let $\g$ be a $(\Z/2\Z)$ graded Lie algebra. Let $x \in \g_1$. Then there exists a graded homomorphsm $\pi_x: sl_2 \to \g$ such that $\pi_x(e)=x.$
\end{lemma}
For a proof see e.g. \cite[Lemma 7.1.11]{AG_HC}.
\begin{remark}
It is easy to see that $\pi_x$ is uniquely defined up to the
exponentiated adjoint action of $(\g_0)_x$.
\end{remark}

\begin{definition}
Let $\g$ be a $(\Z/2\Z)$ graded Lie algebra. Let $x \in \g_1$. We define the defect of $x$ to be the defect of $\g$ considered as a representation of $sl_2$ via $\pi_x$. Clearly it  does not depends on the choice of $\pi_x$
\end{definition}

\begin{lemma} \label{BilForm}
Let $(G,H,\theta)$ be a symmetric pair. Then there exists a
$G$-invariant $\theta$-invariant non-degenerate symmetric bilinear
form $B$ on $\g$. In particular, $B|_{\h}$ and $B|_{\g^{\sigma}}$
are also non-degenerate and $\h$ is orthogonal to $\g^{\sigma}$.
\end{lemma}
For a proof see e.g. \cite[Lemma 7.1.9]{AG_HC}.

From now on we will fix such $B$ and identify $\gd$ with $(\gd)^*.$

\begin{proposition}\label{neg_def_no_dist}
Let $(G,H,\theta)$ be a symmetric pair. Let $\xi \in \Sc^*(Q(\gd))$. Suppose that both  $\xi$ and $\Fou(\xi)$ are supported on  $\Gamma(\gd)$. Then the set of elements in $\Supp(\xi)$ which have non-negative defect is dense in $\Supp(\xi)$
\end{proposition}
The proof is the same as the proof of \cite[Proposition 7.3.7]{AG_HC}.

\subsection{All the nice symmetric pairs are regular}
\begin{definition}
let $(G,H,\theta)$  be a symmetric pair Let $x \in \Gamma(\gd)$ be a nilpotent element. we will call it distinguished if it is distinguished with respect to the action of $H$ on $\gd.$
\end{definition}

\begin{lemma}
let $(G,H,\theta)$  be a symmetric pair. Assume that $\g$ is semi-simple. Then\\
(i) for any $x \in \gd$ we have $CN_{Hx,x}^{\gd}=(\gd)^x$\\
(ii) $Q(\gd)=\gd.$
\end{lemma}
\begin{proof}$ $\\
(i) is trivial.\\
(ii) assume the contrary: there exist $0 \neq x \in \gd$ such that
$Hx=x$. Then $\dim (CN_{Hx,x}^{\gd}) = \dim \gd$, hence
$CN_{Hx,x}^{\gd} = \gd$ which means, $\gd=(\gd)^x$. therefor $x$
lies in the center of $\g$ which is impossible.
\end{proof}
\begin{corollary}
Our definition of distinguished element coincides with the one in
\cite{Sak}. Namely an element $x \in \Gamma(\gd)$ is
distinguished iff $((\g_s)^\sigma)^x$ does not contain
semi-simple elements. Here $\g_s$ is the semi-simple part of $\g.$
\end{corollary}

\begin{definition}
We will call a symmetric pair $(G,H,\theta)$ a pair of negative distinguished defect if all the distinguished elements in $\Gamma(\gd)$ have negative defect.
\end{definition}

\begin{theorem}\label{neg_dist_def}
Let $(G,H,\theta)$ be a symmetric pair of negative distinguished defect. Then it is special.
\end{theorem}
\begin{proof}
Let $\xi \in \Sc^*(Q(\gd))^{H(F)}$ such that both $\xi$ and $\Fou(\xi)$ are supported in $\Gamma(\gd)$. Choose stratification  $$\Gamma(\gd)=X_n \supset X_{n-1} \supset X_{0}=0 \supset X_{-1}= \emptyset$$ such that $X_i-X_{i-1}$ is an $H$-orbit wich is open in $X_i$. We will prove by descending induction that $\xi$ is suported on $X_i$. So we fix $i$ and assume that $\xi$ is suported on $X_i$, our aim is to prove that $\xi$ is suported on $X_{i-1}$. Suppose that $X_i-X_{i-1}$ is non-distinguished. Then by Corollary \ref{non_disting_no_dists} we have $\Supp(\xi) \subset X_{i-1}$. Now suppose that $X_i-X_{i-1}$ is distinguished. Then by Proposition \ref{neg_def_no_dist} we have $\Supp(\xi) \subset X_{i-1}$.
\end{proof}

We will use the notion of nice symmetric pair from \cite{LS}. We
will use the following definition.

\begin{definition}
A symmetric pair $(G,H,\theta)$ is called nice iff the semi simple part of the pair $(\g,\h)$ decomposes, over the algebraic closure, to a product of pairs
of the following types:
\begin{itemize}
\item $(g_1 \oplus g_1, g_1)$, where $g_1$ is a simple Lie algebra
\item $(sl_m, so_m)$
\item $(sl_{2m}, sl_m \oplus sl_m \oplus \g_a)$, where $\g_a$ is the one dimensional Lie algebra.
\item $(sp_{2m}, sl_m \oplus \g_a)$
\item $(so_{2m + k}, so_{m + k} \oplus so_{m})$, for k = 0, 1, 2
\item $(e_6, sp_8)$
\item $(e_6 , sl_6 \oplus sl_2)$
\item $(e_7 , sl_8)$
\item $(e_8 , so_{16})$
\item $(f_4 , sp_6 \oplus sl_2)$
\item $(g_2 , sl_2 \oplus sl_2)$

\end{itemize}
\end{definition}

This notion is motivated by \cite{Sak}, where the following
theorem is proven (see Theorem 6.3).
\begin{theorem}
Let $(G,H,\theta)$  be a nice symmetric pair. Let $\pi: sl_2 \to \g$ be a graded homomorphism such that $\pi(e)$ is distinguished. Consider $\g$ as a graded representation of $sl_2$, decompose it to irreducible representations by $\g= \bigoplus V_{\lambda_i}^{\omega_i}$. Then
$$\sum_{i \text{ s.t. } \omega_i (-1)^{\lambda_i} =-1} (\lambda_i + 2)-\dim(\gd) > 0.$$
\end{theorem}
\begin{corollary}
Any nice symmetric pair is of negative distinguished defect. Thus by Theorem \ref{neg_dist_def} it is special and hence weakly linearly tame and regular.
\end{corollary}

This corollary follows immediately from the theorem using the following lemma and the fact that $\g \cong \g^*$ as a graded representation of $sl_2$
\begin{lemma}
Let $V$ be a graded representation of $sl_2$. Decompose it to irreducible representations by $V= \bigoplus V_{\lambda_i}^{\omega_i}$. Denote
$$\delta(V) :=\sum_{i \text{ s.t. } \omega_i (-1)^{\lambda_i} =-1} (\lambda_i + 2)-\dim(V_1).$$
Then
$$\delta(V)+\delta(V^*)+\de(V)+\de(V^*)=0$$
\end{lemma}
\begin{proof}
This lemma is straightforward computation using Lemma \ref{Defects} and Lemma \ref{Graded_star}.
\end{proof}

\section{A uniform proof of Multiplicity One Theorems for $GL_n$} \label{MOT}
\setcounter{lemma}{0} In this section we indicate a proof of
Multiplicity one Theorems for $GL_n$ which is uniform for all local
fields of characteristic 0. This theorem was proven for the
non-Archimedean case in \cite{AGRS} and for the Archimedean case
in \cite{AG_AMOT} and \cite{SZ}. We will not give all the details
since this theorem was proven before. We will indicate the main
steps and will give the details in the parts which are more
essential. The proof that we present here is based on the ideas from
the previous proofs and uses our partial analog of the
integrability theorem.

Let us first formulate the Multiplicity one Theorems for $GL_n$.

\begin{theorem}\label{MO}
Consider the standard imbedding $\mathrm{GL}_n(F) \hookrightarrow
\mathrm{GL}_{n+1}(F)$. We consider the action of $\mathrm{GL}_n(F)$ on $\mathrm{GL}_{n+1}(F)$ by conjugation.
Then any $\mathrm{GL}_n(F)$-invariant distribution on $\mathrm{GL}_{n+1}(F)$ is invariant with respect
to transposition.
\end{theorem}

It has the following corollary in representation theory.
\begin{theorem}
Let $\pi$ be an irreducible admissible smooth \Fre representation of
$\mathrm{GL}_{n+1}(F)$ and $\tau$  be an irreducible admissible smooth \Fre representation of
$\mathrm{GL}_{n}(F)$. Then
\begin{equation}\label{dim1}
\dim \Hom_{\mathrm{GL}_n(F)}(\pi,\tau) \leq 1.
\end{equation}
\end{theorem}
\subsection{Notation}
$ $\\
\itemize{
\item Let $V:=V_n$ be the standard $n$-dimensional linear space defined over $F$.
\item Let $\sl(V)$ denote the Lie algebra of operators with zero
trace.
\item Denote $X:=X_n:=\sl(V_n) \times V_n \times V_n^*$.
\item Denote $G:=G_n:=\mathrm{GL}(V_n)$.
\item Denote $\g:=\g_n:=\Lie (G_n)=\mathrm{gl}(V_n)$.
\item Let $G_n$ act on $G_{n+1}$, $\g_{n+1}$ and on $\sl(V_n)$ by
$g(A):= gAg^{-1}$.
\item Let  $G$ act on $V \times V^*$ by $g(v,\phi):=(gv,(g^{-1})^{*}\phi)$. This gives rise to an action of $G$ on $X$.
\item Let $\sigma: X \to X$ be given by $\sigma(A,v,\phi)= A^t,\phi^t,v^t$.
\item  We fix the standard trace form on $\sl(V)$ and the standard form on $V \times V^*$.
\item Denote $S:= \{(A,v,\phi) \in X_n | A^n=0 \text{ and } \phi(A^i v)=0 \text{ for any } 0 \leq i \leq n\}$.
\item Note that $S \supset \Gamma(X)$.
\item Denote $S':= \{(A,v,\phi) \in S | A^{n-1} v=(A^*)^{n-1} \phi=0 \}$.
\item Denote
\begin{multline*}
\check{S} := \{((A_1,v_1,\phi_1),(A_2,v_2,\phi_2)) \in X
\times X \, | \, \forall i,j \in \{1,2\} \\
(A_i,v_j,\phi_j) \in S \text{ and } \forall \alpha \in
\mathrm{gl}(V), \alpha(A_1,v_1,\phi_1) \bot (A_2,v_2,\phi_2)\}.
\end{multline*}
\item Note that
\begin{multline*}
\check{S} = \{((A_1,v_1,\phi_1),(A_2,v_2,\phi_2)) \in X \times X \, | \, \forall i,j \in \{1,2\} \\
(A_i,v_j,\phi_j) \in S \text{ and } [A_1,A_2] + v_1 \otimes \phi_2-v_2 \otimes \phi_1 =0\}.
\end{multline*}
\item Denote $$\check{S}' := \{((A_1,v_1,\phi_1),(A_2,v_2,\phi_2)) \in \check{S}| \, \forall i,j \in \{1,2\} (A_i,v_j,\phi_j) \in S' \}.$$
}
\subsection{Reformulation}
$ $\\
A standard use of the Harish-Chandra descent method shows that it is
enough to show that any $G(F)$ invariant distribution on $X(F)$ is
invariant with respect to $\sigma$, moreover it is enough to show
this under the assumption that this is true for distributions on
$(X-S)(F)$. So it is enough to prove the following theorem

\begin{theorem} \label{special}
The action of $G$ on $X$ is special (and hence weakly linearly
tame).
\end{theorem}
\begin{remark}
One can show that this implies that the action of $G_n$ on
$G_{n+1}$ is tame.
\end{remark}

\subsection{Proof of Theorem \ref{special}}

It is enough to show that any distribution $\xi \in
\Sc^*(X(F))^{G(F)}$, such that $\xi, \, \Fou_{V \times V^*}(\xi), \,
\Fou_{sl(V)}(\xi)$ and $\Fou_{X}(\xi)$ are supported on $S(F)$, is
zero.

\begin{lemma}
Let $\xi \in \Sc^*(X(F))^{G(F)}$ such that both $\xi$ and $\Fou_{V \times V^*}(\xi)$ are supported on $S(F).$ Then $\xi$ is supported on $S'(F).$
\end{lemma}

\begin{proof}
This is a direct computation using Propositions \ref{strat},
\ref{Arch_strat} , Theorem \ref{Frob} and Theorem \ref{ArchHom}, and the fact that $S-S' \subset sl(V) \times (V \times 0 \cup 0 \times V^*)$.
\end{proof}
\begin{corollary}
Let $\xi \in \Sc^*(X(F))^{G(F)}$ such that  $\xi, \Fou_{V \times
V^*}(\xi), \Fou_{sl(V)}(\xi)$ and $\Fou_{X}(\xi)$ are supported on
$S(F)$ then $SS(\xi) \subset \check{S}'$.
\end{corollary}

Now the following geometric statement implies Theorem
\ref{special}.
\begin{theorem} [The geometric statement]\label{Geom}
There are no non-empty $X \times X$-weakly coisotropic
subvarieties of $\check{S}'$.
\end{theorem}
\subsection{Proof of the geometric statement}
\begin{notation}
Denote $\check{S}'':=\{((A_1,v_1,\phi_1),(A_2,v_2,\phi_2)) \in \check{S}'| A_1^{n-1}=0\}$.
\end{notation}
By Theorem \ref{non_disting_no_cois} (and Example \ref{group_case}) there are no non-empty $X
\times X$-weakly coisotropic subvarieties of $\check{S}''$.
Therefore it is enough to prove the following Key proposition.
\begin{proposition}[Key proposition]
There are no non-empty $X \times X$-weakly coisotropic
subvarieties of $\check{S}'- \check{S}''.$
\end{proposition}
\begin{notation}
Let $A \in sl(V)$ be a nilpotent Jordan block. Denote $$R_A := (\check{S}'-\check{S}'')|_{\{A\} \times V \times V^*}.$$
\end{notation}
By Proposition \ref{GeoFrob} the Key proposition follows from the
following Key Lemma.
\begin{lemma}[Key Lemma]
There are no non-empty $V \times V^* \times V \times V^*$-weakly
coisotropic subvarieties of $R_A$.
\end{lemma}
\begin{proof}
Denote $Q_A = \bigcup_{i=1}^{n-1} (Ker A^i) \times
(Ker(A^*)^{n-i}).$ It is easy to see that $R_A \subset Q_A \times
Q_A$ and
$$Q_A \times Q_A = \bigcup_{i,j=0}^n (Ker A^i) \times (Ker (A^*)^{n-i}) \times (Ker A^j) \times (Ker (A^*)^{n-j}).$$
Denote $L_{ij} := (Ker A^i) \times (Ker (A^*)^{n-i}) \times (Ker A^j) \times (Ker (A^*)^{n-j})$.

It is easy to see that any weakly coisotropic subvariety of $Q_A \times Q_A$ is contained in $\bigcup_{i=1}^{n-1} L_{ii}.$
Hence it is enough to show that for any $0<i<n$, we have $\dim R_A \cap L_{ii} < 2n$.

Let $f \in {\mathcal O}(L_{ii})$ be the polynomial defined by
$$f(v_1,\phi_1,v_2,\phi_2):= (v_1)_{i} (\phi_2)_{i+1}-(v_2)_{i} (\phi_1)_{i+1},$$
where $( \cdot)_i$ means the i-th coordinate. It is enough to show that $f(R_A \cap L_{ii}) = \{0\}$.

Let $(v_1,\phi_1,v_2,\phi_2) \in L_{ii}$. Let $M:= v_1 \otimes \phi_2-v_2 \otimes \phi_1$. Clearly, $M$ is of the form
$$ M= \begin{pmatrix}
  &0_{i \times i} &* \\
  & 0_{(n-i) \times i} &0_{(n-i) \times (n-i)}
\end{pmatrix}.  $$
Note also that $M_{i,i+1}=f(v_1,\phi_1,v_2,\phi_2)$.

It is easy to see that any $B$ satisfying $[A,B]=M$ is upper triangular. On the other hand, we know that there exists a
nilpotent $B$ satisfying $[A,B]=M$. Hence this $B$ is upper nilpotent, which implies $M_{i,i+1}=0$ and hence $f(v_1,\phi_1,v_2,\phi_2)=0$.

To sum up, we have shown that $f(R_A \cap L_{ii} = \{0\}$, hence $dim(R_A \cap L_{ii})<2n$. Hence every coisotropic subvariety of $R_A$ has dimension less than $2n$ and therefore is empty.
\end{proof}

\end{document}